\documentclass[11pt,letterpaper]{article}

\usepackage{amsthm,amsmath,amssymb}

\usepackage[
letterpaper,
top=0.7in,
bottom=0.9in,
left=1in,
right=1in]{geometry}

\usepackage{hyperref}
\hypersetup{
colorlinks,
colorlinks=true,
urlcolor=blue,
linkcolor=blue,
citecolor=green}

\newtheorem{theorem}{Theorem}[section]
\newtheorem{proposition}[theorem]{Proposition}
\newtheorem{lemma}[theorem]{Lemma}
\newtheorem{corollary}[theorem]{Corollary}

\newtheorem{conjecture}[theorem]{Conjecture}
\theoremstyle{definition}
\newtheorem{definition}[theorem]{Definition}

\theoremstyle{remark}
\newtheorem{remark}[theorem]{Remark}

\providecommand{\cE}{\mathcal E}
\providecommand{\E}{\mathbb E}

\providecommand{\R}{\mathbb R}

\providecommand{\Ent}{\mathrm{Ent}}

\providecommand{\relent}[2]{D\!\left(#1 \,\|\, #2\right)}

\begin{document}

\title{Transport-entropy inequalities and curvature in \\ discrete-space Markov chains}
\author{Ronen Eldan \and James R. Lee \and Joseph Lehec}

\date{}
\maketitle

\begin{abstract}
Let $G = ( \Omega , E )$ be a graph and let $d$ be the graph distance. 
Consider a discrete-time Markov chain $\{Z_t\}$ on $\Omega$ whose kernel $p$ 
satisfies $p(x,y)>0 \Rightarrow \{x,y\} \in E$ for every $x,y\in \Omega$. 
In words, transitions only occur between neighboring points of the graph. 
Suppose further that $(\Omega,p,d)$ has coarse Ricci curvature
at least $1/\alpha$ in the sense of Ollivier:
For all $x,y \in \Omega$, it holds that
$$
W_1(Z_1 \mid \{Z_0 = x\}, Z_1 \mid \{Z_0 = y\}) \leq \left (1-\frac{1}{\alpha}\right ) d(x,y),
$$
where $W_1$ denotes the Wasserstein $1$-distance.

In this note, we derive a transport-entropy inequality:
For any measure $\mu$ on $\Omega$, it holds that
\[
W_1(\mu,\pi) \leq \sqrt{\frac{2\alpha}{2-1/\alpha} \relent{\mu}{\pi}}\,,
\]
where $\pi$ denotes the stationary measure of $\{Z_t\}$
and $\relent{\cdot}{\cdot}$ is the relative entropy.

Peres and Tetali have conjectured a stronger consequence of coarse Ricci curvature,
that a modified log-Sobolev inequality (MLSI) should hold, in analogy with the
setting of Markov diffusions.  We discuss how our approach
suggests a natural attack on the MLSI conjecture.
\end{abstract}


\section{Introduction}

In geometric analysis on manifolds, it is by now well-established that the Ricci curvature of the underlying manifold has profound consequences for functional inequalities and the
rate of convergence of Markov semigroups toward equilibrium.
One can consult, in particular the books \cite{BGL14} and \cite{Villani09}.
Indeed, in the setting of {\em diffusions} (see \cite[\S 1.11]{BGL14}),
there is an elegant theory around the Bakry-Emery {\em curvature-dimension} condition.

Roughly speaking, in the setting of diffusion on a continuous space, when there is
an appropriate ``integration by parts'' formula (that connects the Dirichlet form to the Laplacian),
a positive curvature condition implies powerful functional inequalities.
Most pertinent to the present discussion, positive curvature yields transport-entropy
and logarithmic Sobolev inequalities.

For discrete state spaces, the situation appears substantially more challenging.
There are numerous attempts at generalizing lower bounds on the Ricci curvature
to discrete metric measure spaces; we refer to the upcoming survey \cite{CGMPRSST16}.
At a broad level, these approaches suffer from one of two drawbacks:  Either the
notion of ``positive curvature'' is difficult to verify for concrete spaces,
or the ``expected'' functional analytic consequences do not follow readily.

In the present note, we consider the notion of {\em coarse Ricci curvature} due to Ollivier
\cite{Ollivier09}.  It constitutes an approach of the latter type:  There is a large body of
finite-state Markov chains that have positive curvature in Ollivier's sense, but for many of them
we do not yet know if strong functional-analytic consequences hold.
This study is made more fascinating by the straightforward connection between
coarse Ricci curvature on graphs and the notion of {\em path coupling} arising
in the study of rapid mixing of Markov chains \cite{BD97}.
This is a powerful method to establish fast convergence to the stationary measure; see, for example,
\cite[Ch. 14]{LPW09}.

In particular, if there were an analogy to the diffusion setting that allowed
coarse Ricci curvature lower bounds to yield logarithmic Sobolev inequalities
(or variants thereof), it would even imply new mixing time bounds for well-studied
chains arising from statistical physics, combinatorics, and theoretical computer science.
A conjecture of Peres and Tetali asserts that a {\em modified log-Sobolev inequality} (MLSI) should always hold in this setting.  Roughly speaking, this means that the underlying Markov
chain has exponential convergence to equilibrium in the relative entropy distance.

Our aim is to give some preliminary results
in this direction and to suggest a new approach to establishing MLSI.
In particular, we prove a $W_1$ transport-entropy inequality.
By results of Bobkov and G\"otze~\cite{BG99}, this is equivalent
to a sub-Gaussian concentration estimate for Lipschitz functions.
Sammer has shown that such an inequality follows formally from MLSI
\cite{Sammer05}, thus one can see verification as evidence
in support of the Peres-Tetali conjecture.
Our result also addresses {\sc Problem J} in Ollivier's survey \cite{Ollivier10}.

\subsection{Coarse Ricci curvature and transport-entropy inequalities}

Let $\Omega$ be a countable state space, and let $p: \Omega \times \Omega \to [0,1]$ denote
a transition kernel.
For $x \in \Omega$, we will use the notation $p(x,\cdot)$ to denote the function
$y \mapsto p(x,y)$.
For a probability measure $\pi$ on $\Omega$
and $f : \Omega \to \mathbb{R}_+$,
we define the entropy of $f$ by
\[
\Ent_{\pi}(f) = \E_{\pi} \left[f \log \left( \frac f {\E_\pi [f] } \right) \right] \, .
\]

We also equip $\Omega$ with a metric $d$. 
If $\mu$ and $\nu$ are two probability
measures on $\Omega$, we denote
by $W_1 ( \mu , \nu )$ the transportation cost (or Wasserstein $1$-distance) between
$\mu$ and $\nu$, with the cost function given by the distance $d$.
Namely,
\[
W_1 ( \mu , \nu )  = \inf \left\{ \E \left[ d(X,Y) \right] \right\}
\]
where the infimum is taken on all couplings $(X,Y)$ of $(\mu,\nu)$.
Recall the Monge--Kantorovitch duality formula for $W_1$
(see, for instance,~\cite[Case 5.16]{Villani09}):
\begin{equation}\label{eq:MK}
W_1 ( \mu , \nu )
= \sup  \left\{ \int_\Omega f \, d\mu - \int_\Omega f \, d\nu \right\} ,
\end{equation}
where the supremum is taken over $1$--Lipschitz functions $f$.
We consider the following notion of curvature
introduced by Ollivier~\cite{Ollivier09}.
\begin{definition}
	The \emph{coarse Ricci curvature} of $(\Omega, p, d)$ is the largest $\kappa \in [-\infty , 1 ]$
	such that the inequality
	\[
	W_1 ( p(x, \cdot) , p(y, \cdot) ) \leq (1-\kappa) \, d(x,y)
	\]
	holds true for every $x,y \in \Omega$.
\end{definition}
In the sequel we will be interested in positive
Ricci curvature. Under this condition the map $\mu \mapsto \mu p$
is a contraction for $W_1$. As a result, it has a unique fixed point
and $\mu p^n$ converges to this fixed point as $n \to \infty$.
In other words the Markov kernel $p$ has a unique stationary measure
and is ergodic.
The main purpose of this note is to show that positive curvature yields a transport-entropy inequality, or equivalently a Gaussian concentration inequality for the stationary measure.
\begin{definition}
	We say that a probability measure $\mu$
	on $\Omega$ satisfies the \emph{Gaussian concentration
		property} with constant $C$ if the inequality
	\[
	\int_\Omega \exp( f ) \, d \mu
	\leq \exp \left( \int_\Omega f \, d\mu + C \, \Vert f \Vert_{\text{Lip}}^2 \right)
	\]
	holds true for every Lipschitz function $f$.
\end{definition}

Now we spell out the dual formulation of the Gaussian
concentration property in terms of transport
inequality. Recall first the definition of the relative entropy
(or Kullback divergence): for two measures $\mu,\nu$ on $(\Omega,\mathcal B)$,
\[
\relent{\nu}{\mu} =
\Ent_{\mu}[\tfrac{d \nu}{d \mu} ] = \int_\Omega \log \left( \frac{ d\nu }{ d\mu} \right) \, d\nu
\]
if $\nu$ is absolutely continuous with respect to $\mu$
and $\relent{\nu}{\mu} = +\infty$ otherwise.
As usual, if $X$ and $Y$ are random variables with laws $\nu$ and $\mu$,
we will take $\relent{X}{Y}$ to be synonymous with $\relent{\nu}{\mu}$.
\begin{definition}
	We say that $\mu$ satisfies $(T_1)$
	with constant $C$ if for every
	probability measure $\nu$ on $\Omega$ we have
	\[
	W_1 ( \mu , \nu )^2 \leq C \cdot \relent{\nu}{\mu} . \qquad\quad \mbox{($T_1$)}
	\]
\end{definition}
As observed by Bobkov and G\"otze~\cite{BG99}, the inequality~$(T_1)$
and the Gaussian concentration property are equivalent.
\begin{lemma}\label{lem:duality}
	A probability measure $\mu$ satisfies the Gaussian concentration property
	with constant $C$ if and only if it satisfies $(T_1)$
	with constant $4C$.
\end{lemma}
This is a relatively straightforward consequence of
the Monge--Kantorovitch duality~\eqref{eq:MK};
we refer to~\cite{BG99} for details.
\begin{theorem}\label{thm:main}
	Assume that $(\Omega, p , d)$ has positive coarse Ricci curvature $1/\alpha$
	and that the one--step transitions all satisfy $(T_1)$
	with the same constant $C$:  Suppose that for every $x\in \Omega$ and for every probability measure $\nu$ we have
	\begin{equation}\label{eq:onestep}
	W_1 ( \nu , p(x,\cdot) )^2  \leq C \cdot \relent{\nu}{p(x,\cdot)}\,.
	\end{equation}
	Then the stationary measure $\pi$ satisfies $T_1$
	with constant $\frac{C \alpha} { 2-1/\alpha}$.
\end{theorem}
\begin{remark}
Observe that Theorem~\ref{thm:main} does not assume reversibility.
\end{remark}
The hypothesis~\eqref{eq:onestep}
might seem unnatural at first
sight but it is automatically satisfied
for the random walk on a graph when $d$ is the graph distance.
Indeed, recall Pinsker's inequality: For every probability
measures $\mu,\nu$ we have
\[
\mathrm{TV}(\mu,\nu) 
\leq \sqrt{  \frac 12 \, \relent{\mu}{\nu} } ,
\]
where $\mathrm{TV}$ denotes the total variation distance. 
This yields the following lemma. 
\begin{lemma}
	Let $\mu$ be a probability measure on a metric
	space $(M,d)$ and assume that the support of $\mu$
	has finite diameter $\Delta$.
	Then $\mu$ satisfies $(T_1)$ with constant $\Delta^2/2$.
\end{lemma}
\begin{proof}
	Let $\nu$ be absolutely continuous with respect to
	$\mu$. Then both $\mu$ and $\nu$ are supported
	on a set of diameter $\Delta$. This implies
	that
	\[
	W_1 ( \mu , \nu ) \leq \Delta\cdot \mathrm{TV} ( \mu , \nu ) .
	\]
	Combining this with Pinsker's inequality we get 
        $W_1 ( \mu , \nu )^2 \leq \frac {\Delta^2} 2 \relent{\nu}{\mu}$,
	which is the desired result.
\end{proof}
\medskip
\noindent
{\bf Random walks on graphs.}
A particular case of special interest will be
random walks on finite graphs.  Let $G=(V,E)$ be
a connected, undirected graph, possibly with self-loops.
Given non-negative conductances $c : E \to \R_+$ on the edges,
we recall the Markov chain $\{X_t\}$ defined by
\[
\Pr[X_{t+1} = y \mid X_t = x] = \frac{c(\{x,y\})}{\sum_{z \in V} c(\{x,z\})}\,.
\]
We refer to any such chain as {\em a random walk on the graph $G$.}
If it holds that $c(\{x,x\}) \geq \frac12 \sum_{Z \in V} c(\{x,z\})$ for all $x \in V$,
we say that the corresponding random walk is {\em lazy.}
We will equip $G$ with its graph distance $d$.

In this setting,
the transitions of the walk are supported on a set of diameter $2$.
So combining the preceding lemma with Theorem~\ref{thm:main}, one arrives at the following.
\begin{corollary}\label{cor:main}
	If a random walk on a graph has
	positive coarse Ricci curvature $\tfrac{1}{\alpha}$
(with respect to the graph distance),
then the stationary
	measure $\pi$ satisfies
	\[
	W_1 ( \mu , \pi )^2 \leq \frac {2 \alpha}{2-1/\alpha} \, \relent{\mu}{\pi},
	\]
	for every probability measure $\mu$.
\end{corollary}

\begin{remark} 
One should note that in this context we have 
\[
d(x,y) \leq W_1  \left( p(x,\cdot ) , p(y,\cdot ) \right) + 2 , \quad \forall x,y \in \Omega  ,
\]
just because after one step the walk is at distance 
$1$ at most from its starting point. As a result, 
having coarse Ricci curvature $1/\alpha$ implies that the diameter 
$\Delta$ of the graph is at most $2\alpha$. So by the previous lemma, 
\emph{every} measure on the graph satisfies $T_1$ with constant $2 \alpha^2$. 
The point of Corollary~\ref{cor:main} is that for the stationary measure 
$\pi$ the constant is order $\alpha$ rather than $\alpha^2$. 
\end{remark}
 
\medskip

We now present two proofs of Theorem~\ref{thm:main}.  The first proof is rather short
and based on the duality formula \eqref{eq:MK}.
The second argument provides an explicit coupling
based on an entropy-minimal drift process.
In Section~\ref{sec:mlsi}, we discuss
logarithmic Sobolev inequalities.
In particular, we present a conjecture about the structure
of the entropy-minimal drift that is equivalent to
the Peres-Tetali MLSI conjecture.

\medskip

After the first version of this note was released 
we were notified that Theorem~\ref{thm:main} was proved 
by Djellout, Guillin and Wu in~\cite[Proposition 2.10]{DGW04}. 
Note that this article actually precedes Ollivier's work. 
The proof given there corresponds to our first proof, by duality. 
Our second proof is more original but does share some
similarities with the argument given by K.~Marton in~\cite[Proposition 1]{Marton96}.
Also, after hearing about our work, Fathi and Shu \cite{FS15} used
their transport-information framework to provide yet another proof.

\section{The $W_1$ transport-entropy inequality}

We now present two proofs of Theorem~\ref{thm:main}.
Recall the relevant data $(\Omega,p,d)$.
Define the process $\{B_t\}$ to be the discrete-time random walk on $\Omega$ corresponding
to the transition kernel $p$.
For $x \in \Omega$, we will use $B_t(x)$ to denote the random variable $B_t \mid \{B_0=x\}$.
For $t \geq 0$, we make the definition
\[
P_t[f](x) = \E[f(B_t(x))]\,.
\]

\subsection{Proof by duality}

	Let $f : \Omega \to \R$ be a Lipschitz function.
	Using the hypothesis~\eqref{eq:onestep} and
	Lemma~\ref{lem:duality}
	we get
	\[
	P_1 [ \exp(f) ] (x) \leq \exp\left( P_1[f] (x) + \frac C 4 \Vert f \Vert_{\rm Lip}^2 \right) ,
	\]
	for all $x \in \Omega$. Applying this inequality repeatedly we obtain
	\begin{equation}\label{eq:step}
	P_n [ \exp(f) ](x)  \leq
	\exp \left( P_n [f](x) + \frac C 4 \sum_{k=0}^{n-1} \Vert P_k f \Vert_{\rm Lip}^2 \right) ,
	\end{equation}
	for every integer $n$ and all $x \in \Omega$. Now we use the curvature hypothesis.
	Note that the Monge--Kantorovitch duality~\eqref{eq:MK}
	yields easily
	\begin{align*}
	1-\tfrac{1}{\alpha}
	~& = \sup_{x\neq y} \left\{ \frac { W_1 ( p(x, \cdot) , p(y, \cdot) ) }{ d(x,y) } \right\} \\
	& = \sup_{x\neq y, g} \left\{ \frac { P_1[g](x) - P_1[g](y) }{ \Vert g \Vert_{\rm Lip} d(x,y) } \right  \} \\
	& = \sup_{g} \left\{ \frac{ \Vert P_1[g] \Vert_{\rm Lip}}{ \Vert g \Vert_{\rm Lip} } \right\} .
	\end{align*}
	Therefore $\Vert P_1[g] \Vert_{\rm Lip} \leq (1-1/\alpha) \Vert g \Vert_{\text{Lip}}$
	for every Lipschitz function $g$ and thus
	\[
	\Vert P_n[f] \Vert_{\rm Lip} \leq (1-1/\alpha)^n \Vert f \Vert_{\text{Lip}} ,
	\]
	for every integer $n$.
	Inequality~\eqref{eq:step} then yields
	\[
	P_n\left[\exp(f)\right](x)
	\leq \exp \left( P_n[f](x)  + \frac {C \alpha}{4(2-1/\alpha)}  \Vert f \Vert_{\rm Lip}^2 \right) .
	\]
	Letting $n \to \infty$ yields
	\[
	\int_\Omega \exp(f) \, d \pi
	\leq \exp \left( \int_\Omega f \, d \pi + \frac {C \alpha}{4\,(2-1/\alpha)}  \Vert f \Vert_{\rm Lip}^2 \right) .
	\]
	The stationary measure $\pi$ thus satisfies Gaussian concentration
	with constant $\frac {C \alpha}{4\,(2-1/\alpha)}$. Another application of the duality, Lemma~\ref{lem:duality}, yields the desired outcome, proving Theorem~\ref{thm:main}.

\subsection{An explicit coupling}

As promised, we now present a second proof of Theorem~\ref{thm:main} based on
an explicit coupling.  The proof does not rely on duality,
and our hope is that the method presented will be
useful for establishing MLSI; see Section~\ref{sec:mlsi}.

The first step of the proof follows a similar idea to the one used in \cite[Proposition 1]{Marton96}. Given the random walk $\{B_t\}$ and another process $\{X_t\}$ (not necessarily Markovian), 
there is a natural coupling between the two processes that takes advantage of the curvature condition and gives a bound on the distance between the processes at time $T$ in terms of the relative entropy. This step is summarized in the following result.

\begin{proposition} \label{prop:counpling}
Assume that $(\Omega, p , d)$ satisfies the conditions of Theorem \ref{thm:main}.
Fix a time $T$ and a point $x_0\in M$. Let $\{B_0 =x_0 , B_1 , \dotsc , B_T\}$ 
be the corresponding discrete time random walk
starting from $x_0$ and let $\{X_0=x_0, X_1, \dotsc, X_T\}$ be an arbitrary 
random process on $\Omega$ starting from $x_0$.
Then, there exists a coupling between the processes 
$(X_t)$ and $(B_t)$ such that
$$
E[d(X_T,B_T)] \leq \sqrt{\frac {C \alpha}{ 2 - 1/\alpha } \relent{\{X_0, X_1, \ldots, X_T\}}{\{B_0, B_1, \ldots, B_T\}}}.
$$
\end{proposition}

In view of the above proposition, proving a transportation-entropy inequality for $(\Omega,p,d)$ is reduced to the following: given a measure $\nu$ on $\Omega$, we are looking for a process $\{X_t\}$ which satisfies: (i) $X_T \sim \nu$ and (ii) the relative entropy between $(X_0,\dotsc,X_T)$ and 
$(B_0,\dotsc,B_T)$ is as small as possible.

To achieve the above, our key idea is the construction of a process $X_t$ which is entropy minimal in the sense that it satisfies
\begin{equation}\label{eq:entoptimal}
X_T \sim \nu \mbox{  and  } \relent{\{X_0, X_1, \ldots, X_T\}}{\{B_0, B_1, \ldots, B_T\}} = \relent{X_T}{B_T}.
\end{equation}
This process can be thought of as the Doob transform of the random walk with a given target law. In the setting of Brownian motion on $\mathbb R^n$ equipped with the Gaussian measure, the corresponding process appears in work of F\"ollmer \cite{Follmer88,Follmer86}. See \cite{Lehec13} for applications to functional inequalities, and the work of L\'eonard \cite{Leonard14} for a somewhat different
perspective on the connection to optimal transportation.

\subsubsection{Proof of Proposition \ref{prop:counpling}: Construction of the coupling }

Given $t\in \{1, \dotsc , T\}$ and $x_1,\dotsc , x_{t-1} \in M$, 
let $\nu ( t, x_0 , \dotsc , x_{t-1} ,  \cdot )$ be the conditional 
law of $X_t$ given $X_0=x_0 , \dotsc , X_{t-1} = x_{t-1}$. 
Now we construct the coupling of $X$ and $B$ as follows. 
Set $X_0=B_0=x_0$ and given 
$(X_1,B_1) , \dotsc , (X_{t-1} , B_{t-1})$ set $(X_t , B_t)$ 
to be a coupling of $\nu ( t , X_0 , \dotsc , X_{t-1}, \cdot  )$ and $p ( B_{t-1} , \cdot )$
which is optimal for $W_1$. Then by construction   
the marginals of this process coincide with the original 
processes $\{X_t\}$ and $\{B_t\}$.
 
The next lemma follows from the coarse Ricci curvature property and the definition 
of our coupling.
\begin{lemma}
For every $t \in \{ 1 , \dotsc, T\}$,
\[
\E_{t-1} \left[ d(X_{t},B_{t}) \right]
\leq \sqrt{C \cdot \relent{ \nu ( t , X_0 , \dotsc , X_{t-1} , \cdot )} {p(X_{t-1}, \cdot)} } +
		\left(1-\frac{1}{\alpha}\right) d(X_{t-1},B_{t-1} ) 
\]
where $\E_{t-1}[\cdot]$ stands for the conditional 
expectation given $(X_0,B_0), \dotsc, (X_{t-1} , B_{t-1})$.
\end{lemma}
\begin{proof}
By definition of the coupling, the triangle inequality for $W_1$, 
the one-step transport inequality~\eqref{eq:onestep}
and the curvature condition
\begin{align*}
\E_{t-1}  \left[d(X_{t},B_{t})\right] 
& =  W_1 \left( \nu (t,X_0,\dotsc, X_{t-1} , \cdot ) , p( B_{t-1} , \cdot ) \right) \\ 
& \leq W_1 \left( \nu (t,X_0,\dotsc, X_{t-1} , \cdot ) , p( X_{t-1} , \cdot) \right)  
+ W_1 \left( p ( X_{t-1} , \cdot ) , p ( B_{t-1}, \cdot ) \right)\\
& \leq \sqrt{C \cdot \relent{ \nu (t,X_0,\dotsc, X_{t-1} , \cdot ) }{p(X_{t-1}, \cdot)}}
	+ \left(1-\frac{1}{\alpha}\right) d(X_{t-1}, B_{t-1})\,. \quad \qedhere
\end{align*}
\end{proof}
Remark that the chain rule for relative entropy asserts that
\begin{equation}\label{eq:chainrule}
\sum_{t=1}^T \E[\relent{ \nu (t,X_0,\dotsc, X_{t-1} , \cdot ) }{ p(X_{t-1}, \cdot) } ] =
	\relent{\{X_0, X_1, \ldots, X_T\}}{\{B_0, B_1, \ldots, B_T\}}.
	\end{equation}
Using the preceding lemma inductively and  then Cauchy-Schwarz yields
	\begin{align*}
	\E[d(X_T,B_T)] & \leq \sum_{t=1}^T \left(1-\frac{1}{\alpha}\right)^{T-t} \E\left[ \sqrt{C \cdot \relent{ \nu (t,X_0,\dotsc, X_{t-1} , \cdot ) }{ p(X_{t-1}, \cdot) } }\right ]  \\
	&\leq
	\sqrt{\sum_{t=1}^T \left(1-\frac{1}{\alpha}\right)^{2(T-t)}}
	\sqrt{\sum_{t=1}^T  C\cdot \E\left[\relent{ \nu (t,X_0,\dotsc, X_{t-1} , \cdot ) }{p(X_{t-1}, \cdot)} \right ]} \nonumber\\
	& \stackrel{\eqref{eq:chainrule}}{\leq}
	\sqrt{\frac{\alpha}{2-1/\alpha}}
	\sqrt{C \cdot \relent{\{X_0, X_1, \ldots, X_T\}}{\{B_0, B_1, \ldots, B_T\}}},
	\end{align*}
	completing the proof of Proposition~\ref{prop:counpling}.

\subsubsection{The entropy-optimal drift process}
\label{sec:follmer}

Our goal in this section is to construct a process 
$x_0=X_0,X_1,\dots,X_T$ satisfying equation \eqref{eq:entoptimal}.
Suppose that we are given a measure $\nu$ on $\Omega$ 
along with an initial point $x_0 \in \Omega$ and a time $T \geq 1$.
We define the {\em F\"ollmer drift process} associated to $(\nu,x_0,T)$ as the stochastic
process $\{X_t\}_{t=0}^T$ defined as follows.  

Let $\mu_T$ be the law of $B_T(x_0)$ and denote by $f$ the density of $\nu$ with respect to $\mu_T$. Note that $f$ is well-defined as long as the support of $\mu_T$ is $\Omega$. Now let 
$\{X_t\}_{t=0}^T$ be the non homogeneous Markov chain on $\Omega$ whose transition 
probabilities at time $t$ are given by 
\begin{equation}\label{eq:follmer-discrete}
q_{t}(x,y) := \mathbb P ( X_{t} = y ) \mid X_{t-1} = x ) 
=  \frac{P_{T-t} f (y)}{P_{T-t+1} f(x)}\, p (x,y) . 
\end{equation}
We will take care in what follows to ensure the denominator does not vanish.
Note that $(q_t)$ is indeed a transition matrix as 
\begin{equation}\label{eq:prob}
\sum_{y \in \Omega} P_{T-t} f (y) \, p (x,y)  = P_{T-t+1} f (x ) . 
\end{equation}

We now state a key property of the drift.
\begin{lemma}\label{lem:follmer-density}
If $p^T(x_0,x) > 0$ for all $x \in \mathrm{supp}(\mu)$, then
$\{X_t\}$ is well-defined.
Furthermore, for every $x_1 , \dotsc , x_T \in \Omega$ we have 
\begin{equation} \label{eq:com}
\mathbb P \left( (X_1,\ldots,X_T)=(x_1,\ldots,x_T) \right)
= \mathbb P \left( (B_1,\ldots,B_T)=(x_1,\ldots,x_T) \right) \, f(x_T) . 
\end{equation}
In particular $X_T$ has law $d \nu = f \, d\mu_T$.
\end{lemma}
\begin{proof}
By definition of the process $(X_t)$ we have 
\[
\begin{split}
\mathbb P \left( (X_1,\ldots,X_T)=(x_1,\ldots,x_T) \right)
& = \prod_{t=1}^T \mathbb P 
\left( X_t = x_t \mid  (X_1,\ldots,X_{t-1} )=(x_1,\ldots,x_{t-1} ) \right) \\
& = \prod_{t=1}^T \frac{ P_{T-t} f (x_t) }{ P_{T-t+1} f ( x_{t-1} } p( x_{t-1},x_t) \\
& = \frac{ f (x_T) }{ P_T f (x_0) } \left(\prod_{t=1}^T p(x_{t-1},x_t)\right)  ,
\end{split}
\]
which is the result. 
\end{proof}
In words, the preceding lemma asserts that the law of the process 
$\{ X_t\}$ has density $f ( x_T )$ with respect to the law of the 
process $\{ B_t \}$. As a result we have in particular 
\[
\relent{\{X_0, X_1, \ldots, X_T\}}{\{B_0, B_1, \ldots, B_T\}} = 
\E [ \log f ( X_T ) ] = \relent{ \nu }{ \mu_T } ,  
\]
since $X_T$ has law $\mu$. 
Note that for any other process $\{Y_t\}$ such that $Y_0 = x_0$ and $Y_T$ has law $\nu$, 
one always has the inequality
\begin{equation}\label{eq:inequality}
\relent{\{Y_0, Y_1, \ldots, Y_T\}}{\{B_0, B_1, \ldots, B_T\}} \geq \relent{Y_T}{B_T} 
= \relent { \nu }{ \mu_T } . 
\end{equation}
Besides $\{X_t\}$ is the unique random process for which this inequality is tight.
Uniqueness follows from strict convexity of the relative entropy.

We summarize this section in the following lemma.
\begin{lemma} \label{lem:follmer}
Let $(\Omega, p)$ be a Markov chain. Fix $x_0 \in \Omega$ and let $x_0 = B_0, B_1, \dots$ be the associated random walk. Let $\nu$ be a measure on $\Omega$ and let $T>0$ be such that for any $y \in \Omega$ one has that $\mathbb{P}(B_T = y) > 0$. Then there exists a process $x_0 = X_0, X_1, \dots, X_T$ such that:
\begin{itemize}
\item
$\{X_0,\dotsc, X_T\}$ is a (time inhomogeneous) Markov chain.
\item
$X_T$ is distributed with the law $\nu$.
\item
The process satisfies equation \eqref{eq:entoptimal}, namely
$$
\relent{X_T}{B_T} = \relent{\{X_0, X_1, \ldots, X_T\}}{\{B_0, B_1, \ldots, B_T\}}.
$$
\end{itemize}
\end{lemma}

\subsection{Finishing up the proof}

Fix an arbitrary $x_0 \in \Omega$ and consider some $T \geq \mathrm{diam}(\Omega,d)$.
Let $\{X_t\}$ be the F\"ollmer drift process associated to the initial data $(\nu,x_0,T)$.
Then combining Proposition \ref{prop:counpling} and equation \eqref{eq:entoptimal},
we have
\begin{equation}\label{eq:pre-tran}
W_1(\nu, \mu_T) 
= \E[d (X_T, B_T) ] \leq \sqrt{\frac{C\alpha}{2-1/\alpha} \,\relent{\nu}{\mu_T} }.
\end{equation}
Now let $T \to \infty$ so that $\mu_T \to \pi$,	yielding the desired claim.

\section{The Peres-Tetali conjecture and log-Sobolev inequalities}
\label{sec:mlsi}

Recall that $p : \Omega \times \Omega \to [0,1]$ is
a transition kernel on the finite state space $\Omega$
with a unique stationary measure $\pi$.
Let $L^2(\Omega,\pi)$ denote the space of real-valued
functions $f : \Omega \to \R$ equipped with the
inner product $\langle f,g\rangle = \E_{\pi}[fg]$.
From now on we assume that the measure $\pi$ is 
reversible, which amounts to saying that the operator
$f \mapsto p f$ is self-adjoint in $L^2 ( \Omega ,\pi )$. 

We define the associated Dirichlet form
\[
\cE(f,g) = \langle f , (p-I)  g \rangle = 
\tfrac12 \sum_{x,y \in \Omega} \pi(x) p(x,y) (f(x) - f(y)) (g(x)-g(y)) \,.
\]
Recall the definition of the entropy of a function $f : \Omega \to \R_+$:
\[
\Ent_{\pi}(f) = \E_{\pi} \left[ f \log \left( \frac{f}{\E_{\pi} f} \right) \right]\,.
\]

Now define the quantities
\begin{align*}
\rho &= \inf_{f : \Omega \to \R_+} \frac{\cE(\sqrt{f},\sqrt{f})}{\Ent_{\pi}(f)} \\
\rho_0 &= \inf_{f : \Omega \to \R_+} \frac{\cE(f,\log f)}{\Ent_{\pi}(f)}\,.
\end{align*}
These numbers are called, respectively, 
the {\em log-Sobolev} and {\em modified log-Sobolev} constants of
the chain $(\Omega,p)$.
We refer to \cite{MT06} for a detailed discussion of such inequalities 
on discrete-space Markov
chains and their relation to mixing times.

One can understand both numbers as measuring the rate of convergence to equilibrium
in appropriate senses.  The modified log-Sobolev constant, in particular,
can be equivalently characterized as the largest value $\rho_0$ such that
\begin{equation}\label{eq:isequiv}
\Ent_{\pi} (H_t f) \leq e^{-\rho_0 t} \Ent_{\pi}(f)
\end{equation}
for all $f : \Omega \to \R_+$ and $t > 0$ (see \cite[Prop. 1.7]{MT06}).
Here, $H_t : L^2(\Omega,\pi) \to L^2(\Omega,\pi)$ is the heat-flow operator
associated to the {\em continuous-time} random walk, i.e., $H_t = e^{-t (I-P)}$,
where $P$ is the operator defined by $P f(x) = \sum_{y \in \Omega} p(x,y) f(y)$.

The log-Sobolev constant $\rho$ controls the hypercontractivity 
of the semigroup $(H_t)$, which in turn yields a stronger notion of 
convergence to equilibrium; again see~\cite{MT06} for a precise statement. 
Interestingly, in the setting of diffusions,
there is no essential distinction between the two notions; one should
consider the following calculation only in a formal sense:
\[
``\ \cE(f, \log f) = \int \nabla f \nabla \log f = \int \frac{|\nabla f|^2}{f} = 4 \int \left|\nabla \sqrt{f}\right|^2 = 4\, \cE\!\left(\sqrt{f},\sqrt{f}\right)\,.\ ''
\]
However, in the discrete-space setting, the tools of differential calculus
are not present.  Indeed, one has the bound $\rho \leq 2 \rho_0$ \cite[Prop 1.10]{MT06}, but there
is no uniform bound on $\rho_0$ in terms of $\rho$.

\subsection{MLSI and curvature}

We are now in position to state an important conjecture
linking curvature and the modified log-Sobolev constant;
it asserts that on spaces with positive coarse Ricci curvature,
the random walk should converge to equilibrium exponentially
fast in the relative entropy distance.

\begin{conjecture}[Peres-Tetali, unpublished]
\label{conj:pt}
Suppose $(\Omega,p)$ corresponds to lazy random walk on a finite graph and $d$ is the graph distance.  If $(\Omega,p,d)$ has coarse Ricci curvature $\kappa > 0$, then the modified log-Sobolev constant satisfies
\begin{equation}\label{eq:pt}
\rho_0 \geq C \kappa\,.
\end{equation}
where $C > 0$ is a universal constant.
\end{conjecture}

A primary reason for our interest in Corollary~\ref{cor:main} is that, by results of
Sammer \cite{Sammer05}, Conjecture~\ref{conj:pt} implies Corollary~\ref{cor:main}.
We suspect that a stronger conclusion should hold in many cases;
under stronger assumptions, it should be that one can obtain a lower bound
on the (non-modified) log-Sobolev constant $\rho \geq C \kappa$.
See, for instance, the beautiful approach of Marton \cite{Marton15}
that establishes a log-Sobolev inequality for product spaces assuming
somewhat strong contraction properties of the Gibbs sampler.

However, we recall that this cannot hold under just 
the assumptions of Conjecture~\ref{conj:pt}.
Indeed, if $G=(V,E)$ is the complete graph on $n$ vertices, it is easy to see that the
coarse Ricci curvature $\kappa$ of the lazy random walk is $1/2$.
On the other hand, one can check that the log-Sobolev constant $\rho$ 
decays asymptotically like $\frac{1}{\log n}$ (use the test function $f = \delta_x$ for some fixed $x \in V$).

\subsection{An entropic interpolation formulation of MLSI}
\label{sec:entropy-interp}

We now suggest an approach to Conjecture~\ref{conj:pt} using an entropy-optimal drift process.
While we chose to work with discrete-time chains in Section~\ref{sec:follmer},
working in continuous-time will allow us more precision in 
exploring Conjecture~\ref{conj:pt}.
We will use the notation introduced at the beginning of this section.

\medskip
\noindent
{\bf A continuous-time drift process.}
Suppose we have some initial data $(f,x_0,T)$
where $x_0 \in \Omega$ and $f : \Omega \to \R_+$
satisfies $\E_{\pi}[f] = 1$.
Let $\{B_t : t \in [0,\infty)\}$ denote the continuous-time random walk
with jump rates $p$ on the discrete state space $\Omega$ 
starting from $x_0$. We let $\mu_T$ be the law of $B_T$ and let 
$\nu$ be the probability measure defined by 
\[
d \nu =  \frac{ f }{ H_T f ( x_0 ) } \, d\mu_T ,  
\]
where $(H_t)$ is the semigroup associated to the jump rates $p(x,y)$. 
Note that $\nu$ is indeed a probability measure as $\int f \, d\mu_T = H_T f ( x_0 )$
by definition of $\mu_T$. 

We now define the continuous-time F\"ollmer drift process associated to 
the data $(x_0, T , f)$ as the (time inhomogeneous) Markov chain $\{X_t , \, t \leq T\}$
starting from $x_0$ and having transition rates at time $t$ given by
\begin{equation}\label{eq:follmer-continuous}
 q_t(x,y) = p(x,y) \frac{H_{T-t} f(y)}{H_{T-t} f(x)}\,.
\end{equation}
Informally this means that the conditional probability that the process 
$\{X_t\}$ jumps from $x$ to $y$ between time $t$ and $t+dt$ given the past  
is $q_t(x,y) dt$. This should be thought as the continuous-time analogue
of the discrete F{\"o}llmer process defined by~\eqref{eq:follmer-discrete}. 
We claim that again the law of the process $\{ X_t , \, t \leq T \}$ 
has density $f ( x_T ) / H_T f ( x_0 )$ with respect to the law  
of $\{ B_t , \, t \leq T \}$. Let us give a brief justification 
of this claim. 
Define a new probability measure $\mathbb Q$ by setting
\begin{equation}\label{eq:defQ}
\frac{ d \mathbb Q } { d \mathbb P } = \frac{ f ( B_T ) }{ H_T f ( x_0 ) } . 
\end{equation}
We want to prove that the law of $B$ under $\mathbb Q$ coincides
with the law of $X$ under $\mathbb P$. 
Let $(\mathcal F_t)$ be the natural 
filtration of the process $(B_t)$, let $t\in [0,T)$ and let $y\in M$. 
We then have the following computation: 
\[
\begin{split}
\mathbb Q ( B_{t+\Delta t} = y \mid \mathcal F_t ) 
& = \frac { \E^{\mathbb P} [  f( B_T ) \, \mathbf 1_{ \{ B_{t+\Delta t} = y \} }  
\mid \mathcal F_t ] } { \E^{\mathbb P} [ f ( B_T ) \mid \mathcal F_t ] }  + o(\Delta t)\\
& = \frac { H_{T-t} f ( y ) }{ H_{T-t} f ( B_t ) }
\, \mathbb P ( B_{t+\Delta(t)} = y \mid \mathcal F_t ) + o(\Delta t)  \\
&= \frac { H_{T-t} f ( y ) }{ H_{T-t} f ( B_t ) } \, p ( B_t , y ) \, \Delta t + o(\Delta t) . 
\end{split}
\] 
This shows that under $\mathbb Q$, the process $\{B_t , \, t \leq T \}$
is Markovian (non homogeneous) with jump rates at time $t$ given 
by~\eqref{eq:follmer-continuous}. Hence the claim.
 
This implies in particular that $X_T$ has law $\nu$. 
This also yields the following formula for the relative entropy of $\{X_t\}$: 
\begin{equation}\label{eq:opt3}
\relent{ \{ X_t , \, t \leq T \} }{  \{ B_t , \, t \leq T \} } 
= \E \left[ \log \frac{f ( X_T )}{ H_T f ( x_0 ) } \right] = \relent { \nu } { \mu_T }  .
\end{equation}

The process $\{X_t\}$ starts from $x_0$ and has law $\nu$ at time $T$. 
Because $X_T$ has law $\nu$ and $B_T$ has law $\mu_T$, the two processes
must evolve differently.  One can think of the process $\{X_t\}$ as ``spending information''
in order to achieve the discrepancy between $X_T$ and $B_T$.  The amount of
information spent must at least account for the difference in laws at the endpoint, i.e.,
\[
\relent{\{X_t , \, t \leq T\} }{ \{ B_T , \, t \leq T \} } 
\geq \relent{X_T}{B_T} . 
\]
As pointed out in Section~\ref{sec:follmer}, the content
of \eqref{eq:opt3} is that $\{X_t\}$ spends exactly this minimum amount.

For $0 \leq s \leq s'$, we use the notations $B_{[s,s']} = \{ B_t : t \in [s,s']\}$ and
$X_{[s,s']} = \{ X_t : t \in [s,s'] \}$ for the corresponding trajectories. 
From the definition of $\mathbb Q$ we easily get
\begin{equation}\label{eq:change2}
\frac{d \mathbb Q}{d \mathbb P} \Big|_{\mathcal{F}_t} = 
\E \left[ \frac{ f ( B_T ) }{ H_T f(x_0 ) } \mid \mathcal F_t \right] = 
\frac{ H_{T-t} f ( B_t ) }{ H_T f ( x_0 ) } . 
\end{equation}
As a result 
\begin{equation}\label{eq:relent3}
\relent{X_{[0,t]}}{B_{[0,t]}} 
 = \E \left[ \log \frac{ H_{T-t} f ( X_t )}{ H_T f ( x_0 ) } \right]
\end{equation}
for all $t \leq T$. 
Let us now define the rate of information spent at time $t$:
\[
I_t = \frac{d}{dt} \relent{X_{[0,t]}}{B_{[0,t]}}\,.
\]
Intuitively, the entropy-optimal process $\{X_t\}$
will spend progressively more information as $t$ approaches $T$.
Information spent earlier in the process is less valuable
(as the future is still uncertain).
Let us observe that a formal version of this statement
for random walks on finite graphs
is equivalent to Conjecture~\ref{conj:pt}.
\begin{conjecture}\label{conj:mlsi}
Suppose $(\Omega,p)$ corresponds to a lazy random walk on a finite graph and $d$ is the graph distance, and that $(\Omega,p,d)$ has coarse Ricci curvature $1/\alpha$.
Given $f : \Omega \to \R_+$ with $\E_{\pi}[f]=1$ and $x_0 \in \Omega$,
for all sufficiently large times $T$, it holds that
if $\{X_t : t \in [0,T]\}$ is the associated continuous-time
F\"ollmer drift process
process with initial data $(f,x_0, T)$, then
\begin{equation}\label{eq:conj}
   \relent{X_T}{B_T} \leq C \alpha I_T \,,
\end{equation}
where $C > 0$ is a universal constant.
\end{conjecture}

As $T \to \infty$, we have $P_T f (x_0) \to \E_\pi f = 1$ and thus 
\[
\relent{X_T}{B_T} \to  \Ent_\pi ( f )
\]
Moreover, we claim that $I_T \to \cE(f, \log f)$ as $T \to \infty$.
Together, these show that Conjectures \ref{conj:pt} and \ref{conj:mlsi} are equivalent.

To verify the latter claim, note that
from \eqref{eq:relent3} and~\eqref{eq:change2} we have
\[
\begin{split}
\relent{X_{[0,t]}}{B_{[0,t]}} 
& = \E\left[ \log \left( \frac{ H_{T-t} f (X_t) }{ H_T f ( x_0 ) } \right) \right] \\
& = \E\left[ \log \left( \frac{ H_{T-t} f (B_t) }{ H_T f ( x_0 ) } \right) 
\frac{ H_{T-t} f(B_t) }{ H_T f ( x_0 ) } \right] \\
& = 
\frac 1 { H_T f(x_0)} \, H_t \left( H_{T-t} f \log H_{T-t} f \right) ( x_0 ) 
- \log H_T f ( x_0 ).
\end{split}
\]
Differentiating at $t=T$ yields
\[
I_T = \frac 1 { H_T f (x_0) }  \, H_T \left( \Delta ( f \log f ) 
- (\Delta f) (\log f + 1) \right)(x_0) .
\]
where $\Delta = I-p$ denotes the generator of the semigroup $(H_t)$. 
Recall that $\delta_{x_0} H_T$ converges weakly to $\pi$, and that by stationarity 
$\E_\pi \Delta g = 0$ for every function $g$. Thus 
\[
\lim_{T \to \infty} I_T = - \E_\pi [ (\Delta f) \log f ] . 
\]
The latter equals $\cE(f,\log f)$ by reversibility, hence the claim. 

\bibliographystyle{alpha}
\bibliography{transport}

\end{document}